\documentclass[11pt,twoside]{amsart}
\usepackage{latexsym,amssymb,amsmath}

\numberwithin{equation}{section}
\newtheorem{theorem}{Theorem}[section]
\newtheorem{lemma}[theorem]{Lemma}
\newtheorem{proposition}[theorem]{Proposition}
\newtheorem{corollary}[theorem]{Corollary}
\newtheorem{conjecture}[theorem]{Conjecture}

\theoremstyle{definition}
\newtheorem{definition}[theorem]{Definition} 
 
\newtheorem{remark}[theorem]{Remark}
\newtheorem{example}[theorem]{Example}

\begin{document}

%%%%%%%%%%%%%%%%%%%%%%%%%%%%%%%%%%%%%%%%%%%%%%%%%%%%%%%%%%%%%%%%%%%%%

\title[]{The degree and regularity of vanishing ideals of algebraic
toric sets over finite fields} 

\thanks{The first author is a member of the Center for Mathematical Analysis,
Geometry and Dynamical Systems. The second author was partially supported by SNI}

\author{Maria Vaz Pinto}
\address{
Departamento de Matem\'atica\\
Instituto Superior Tecnico\\
Universidade T\'ecnica de Lisboa\\ 
Avenida Rovisco Pais, 1\\ 
1049-001 Lisboa, Portugal 
}\email{vazpinto@math.ist.utl.pt}

\author{Rafael H. Villarreal}
\address{
Departamento de
Matem\'aticas\\
Centro de Investigaci\'on y de Estudios
Avanzados del
IPN\\
Apartado Postal
14--740 \\
07000 Mexico City, D.F.
}
\email{vila@math.cinvestav.mx}
%\urladdr{http://www.math.cinvestav.mx/$\sim$vila/}


\begin{abstract} 
Let $X^*$ be a subset of an affine space
$\mathbb{A}^s$, over a
finite field $K$, which is parameterized by the edges of a clutter.
Let $X$ and $Y$ be the images of $X^*$ under the maps $x\mapsto [x]$
and $x\mapsto[(x,1)]$ respectively, where $[x]$ and $[(x,1)]$ are
points in the projective 
spaces $\mathbb{P}^{s-1}$ and $\mathbb{P}^s$ respectively. 
For certain clutters and for connected
graphs, we were able to relate the algebraic invariants and
properties of the vanishing 
ideals $I(X)$ and $I(Y)$. In a number of interesting cases, we
compute its degree and 
regularity. For Hamiltonian bipartite graphs, we show 
the Eisenbud-Goto regularity conjecture. We give optimal bounds for
the regularity when the graph 
is bipartite. 
It is shown that $X^*$ is an affine torus if
and only if $I(Y)$ is a complete intersection. We present some
applications to coding theory and show some bounds for the minimum distance 
of parameterized linear codes for connected bipartite graphs. 
\end{abstract}

\maketitle 

\section{Introduction}\label{intro-rs-codes}

Let $K=\mathbb{F}_q$  be a finite field with $q\neq 2$ elements and 
let $y^{v_1},\ldots,y^{v_s}$ be a finite set of monomials.  
As usual if $v_i=(v_{i1},\ldots,v_{in})\in\mathbb{N}^n$, 
then we set 
$$
y^{v_i}=y_1^{v_{i1}}\cdots y_n^{v_{in}},\ \ \ \ i=1,\ldots,s,
$$
where $y_1,\ldots,y_n$ are the indeterminates of a ring of 
polynomials with coefficients in $K$. Consider the following sets 
parameterized  by these monomials: (a) the  {\it affine algebraic
toric set\/} 
$$
X^*:=\{(x_1^{v_{11}}\cdots x_n^{v_{1n}},\ldots,x_1^{v_{s1}}\cdots
x_n^{v_{sn}})\in\mathbb{A}^{s}	\vert\, x_i\in K^*\mbox{ for all }i\},
$$
where $K^*=\mathbb{F}_q^*=\mathbb{F}_q\setminus\{0\}$ and
$\mathbb{A}^{s}=K^s$ is an affine 
space over the field $K$, (b) the {\it projective algebraic toric set} 
$$
X:=\{[(x_1^{v_{11}}\cdots x_n^{v_{1n}},\ldots,x_1^{v_{s1}}\cdots
x_n^{v_{sn}})]\, \vert\, x_i\in K^*\mbox{ for all
}i\}\subset\mathbb{P}^{s-1},
$$
where $\mathbb{P}^{s-1}$ is a projective space over the field $K$,
and (c) 
the {\it projective closure\/} of $X^*$
$$
Y:=\{[(x_1^{v_{11}}\cdots x_n^{v_{1n}},\ldots,x_1^{v_{s1}}\cdots
x_n^{v_{sn}},1)]\, \vert\, x_i\in K^*\mbox{ for all
}i\}\subset\mathbb{P}^{s}.
$$

Notice that  $Y$ is parameterized  by
$y^{v_1},\ldots,y^{v_s},y^{v_{s+1}}$, where $v_{s+1}=0$. These three
sets are multiplicative groups under componentwise multiplication. We
are interested in the algebraic invariants (regularity, degree, Hilbert
series)---and in the complete
intersection property---of
the vanishing ideals of these sets. 

Let $S=K[t_1,\ldots,t_s]=\oplus_{d=0}^\infty S_d$ and
$S[u]=\oplus_{d=0}^\infty S[u]_d$ be polynomial rings over the field
$K$ with the standard grading, where $S[u]$ is obtained from $S$ by
adjoining a new variable $u=t_{s+1}$. 

Recall that the {\it vanishing ideal\/} 
of $X^*$, denoted by $I({X^*})$, is the
ideal of $S$ generated by all polynomials that vanish on
${X^*}$. The {\it vanishing ideal\/} of $X$ (resp. $Y$), denoted by
$I(X)$ (resp. $I(Y)$), is the ideal 
of $S$ (resp. $S[u]$) generated by the homogeneous polynomials
that vanish on $X$ (resp. $Y)$. 

In this paper we uncover some relationships between the 
algebraic invariants---and the complete intersection properties---of
$I(X)$ and $I(Y)$. We 
focus on vanishing ideals of algebraic toric sets that are
parameterized by monomials $y^{v_1},\ldots,y^{v_s}$ arising from the 
edges of a graph $G$ or a clutter $\mathcal{C}$ (a clutter is a sort
of hypergraph, see 
Definition~\ref{clutter-def}).

This paper is motivated by the study of parameterized linear codes
\cite{algcodes}, and specifically by the fact that the degree and the Hilbert
function of $S[u]/I(Y)$ are related to the basic parameters of 
parameterized affine linear codes \cite{affine-codes} (see
Theorem~\ref{bridge-affine-projective}). 

The contents of this paper are as follows. In
Section~\ref{degree-and-reg} we study the degree and regularity of
vanishing ideals. It is well known that $|X|$ and $|Y|$ are the
degrees of $S/I(X)$ and $S[u]/I(Y)$ respectively \cite{harris}. We show that
$|Y|\leq (q-1)|X|$ and 
give sufficient conditions for equality in terms of $q$ and the
combinatorics of $\mathcal{C}$ (see
Proposition \ref{maria-cafeteria-cinvestav-gen}). If $G$ is a graph, 
we express $|Y|$ as a function of $q$, $n$ and
$|X|$ (see Theorem~\ref{maria-cafeteria-cinvestav}). For connected
graphs, we express $|Y|$ as a function of $q$ and $n$ only 
(Corollary~\ref{maria-cafeteria-cinvestav-1}). In general the ideal
$I(X)+(t_1^{q-1}-u^{q-1})$ is contained in $I(Y)$. We give 
sufficient conditions for equality (see
Theorem~\ref{sunday-jun19-11}), for instance 
equality occurs 
if $G$ is a bipartite graph or if $G$ is any graph and $q$ is even (see
Corollary~\ref{sunday-jun19-11-night}). It turns out that the
invariants of $S/I(X)$ and $S[u]/I(Y)$ are closely related if equality
occurs (see Proposition~\ref{jul5-11}). For connected bipartite
graphs, 
we give optimal upper and lower bounds for the regularity of
$S/I(X)$ (see Theorem~\ref{upper-lower-bounds-reg-bip}). 
Then, we compute the regularity of any Hamiltonian bipartite
graph (see Corollary~\ref{jun11-11-3}). As a byproduct, we show 
the Eisenbud-Goto regularity conjecture when $G$ is a
Hamiltonian bipartite graph (see
Corollary~\ref{eisenbud-goto-hamiltonian}). Let $X'$ be the set
parameterized by $y^{v_1},\ldots,y^{v_{s-1}}$. If $y_n$ occurs 
only in the monomial $y^{v_s}$, we relate the degree and the
regularity of $I(X)$ and $I(X')$ (see
Theorem~\ref{june-2011--dictamina}). For connected bipartite graphs,
this leads to an improved upper bound 
for the regularity of $S/I(X)$, in terms of the length of a largest
cycle (see
Corollary~\ref{jul5-11-1}).

In Section~\ref{applications}, we give applications
to coding theory, and explain the well known connections between 
the {\it algebraic invariants\/} of vanishing ideals (Hilbert
function, degree, regularity) and the {\it parameters} of affine and
projective parameterized linear codes (dimension, length, minimum
distance). We present upper and lower bounds for
the minimum distance of parameterized codes arising from connected
bipartite graphs (see Theorem~\ref{jun25-11}). 
The bounds are in terms of the minimum distance of
parameterized codes over projective tori. These bounds can be computed using a
recent result of 
\cite{ci-codes} (see Theorem~\ref{maria-vila-hiram-eliseo}). Let
$\delta_Y(d)$ (resp. $\delta_X(d)$) be the minimum distance of the
parameterized projective code of degree $d$ on the set $Y$ (resp.
$X)$, see Definition~\ref{def-param-proj-code}. 
For certain clutters we show that $\delta_Y(d)\leq (q-1)\delta_X(d)$ for
$d\geq 1$ with equality if $d=1$ and $G$ is a connected bipartite graph
(see Proposition~\ref{jun12-11}). 

In Section~\ref{ci-clutters-affine}, we characterize when $I(Y)$ is
a complete intersection in algebraic and geometric terms (see
Theorem~\ref{ci-affine}).  
A result of  \cite{ci-codes} shows 
that $I(X)$ is a complete intersection if and only if $X$ is a
projective torus (see Definition~\ref{projectivetorus-def}). We
complement 
this result by showing that $I(Y)$ is
a complete intersection if and only if $X^*$ is an affine torus 
(see Theorem~\ref{ci-affine}). For connected graphs,  
the complete intersection property of $I(X)$ is
independent of $q$ (see Proposition~\ref{jul27-11}), while 
the complete intersection property of $I(Y)$ depends on $q$. 
We describe when $I(Y)$ is a complete
intersection in terms of $q$ and the combinatorics of the graph 
(see Theorem~\ref{main-iy-graphs}). 

For all unexplained
terminology and additional information  we refer to
\cite{Mats} (for the general theory of commutative rings),   
\cite{AL,Sta1} (for the theory of Gr\"obner bases and Hilbert
functions), \cite{gold-little-schenck,GRT,tsfasman} (for
the theory of Reed-Muller codes and evaluation codes), 
\cite{algcodes} (for the theory of parameterized codes), and
\cite{Boll,cornu-book} (for graph theory and clutter theory). 

\section{The degree and the regularity of vanishing
ideals}\label{degree-and-reg}

We continue to use the notation and definitions used in the
introduction. In this section we study the degree and the regularity
of $S/I(X)$ and $S[u]/I(Y)$. 

\begin{definition}\label{clutter-def}
A {\it clutter\/} $\mathcal{C}$ is a family $E$ of subsets of a
finite ground set $\{y_1,\ldots,y_n\}$ such that if $f_1, f_2 \in
E$, then $f_1\not\subset f_2$. The ground set is called the {\em
vertex set} 
of $\mathcal{C}$ and $E$ 
is called the {\em edge set} of $\mathcal{C}$, they are denoted by
$V_\mathcal{C}$ 
and $E_\mathcal{C}$  respectively. 
\end{definition}

Clutters are special hypergraphs. One 
important example of a clutter is a graph with the vertices and edges
defined in the usual way for graphs \cite{Boll}. 

\begin{definition}\label{charvec-def}
Let $\mathcal{C}$ be a clutter with vertex set
$V_\mathcal{C}=\{y_1,\ldots,y_n\}$ and let $f$ be an edge of
$\mathcal{C}$. The {\it characteristic vector\/} of $f$ is the vector
$v=\sum_{y_i\in f}e_i$, where
$e_i$ is the $i${\it th} unit vector in $\mathbb{R}^n$. 
\end{definition}

Throughout this paper $\mathcal{C}$ will denote a clutter with $n$
vertices and $s$ edges. We will always assume that
$\{v_1,\ldots,v_s\}$ is the 
set of all characteristic vectors of the edges of $\mathcal{C}$. We
also assume that $y_1,\ldots,y_n$ are the vertices 
of $\mathcal{C}$. When $\mathcal{C}$ is a graph, we denote
$\mathcal{C}$ 
by $G$. 

\begin{definition} Let $\mathcal{C}$ be a clutter. We call $X$ (resp.
$X^*$) the {\it projective algebraic toric set} (resp. {\it 
affine algebraic toric set}) {\it parameterized by the edges} of $\mathcal{C}$
\end{definition}

\begin{definition} The {\it Hilbert function\/} of
$S[u]/I(Y)$ is given by 
$$H_Y(d):=\dim_K\, 
(S[u]/I(Y))_d=\dim_K\, 
S[u]_d/I(Y)_d, 
$$
where $I(Y)_d=S[u]_d\cap I(Y)$ is the degree $d$ part of $I(Y)$. 
\end{definition}

The ideal $I(Y)$ is Cohen-Macaulay of height $s$
\cite{geramita-cayley-bacharach}. Thus, $\dim\, S[u]/I(Y)=1$. 
The unique polynomial $h_Y(t)\in \mathbb{Z}[t]$ such
that $h_Y(d)=H_Y(d)$ for 
$d\gg 0$ is called the {\it Hilbert polynomial\/} of  
$S[u]/I(Y)$. In our situation $h_Y(t)$ is a constant.
Furthermore $H_Y(d)=|Y|$ for $d\geq |Y|-1$, see \cite[Lecture
13]{harris}. This means that $|Y|$ is the {\it degree\/} 
of $S[u]/I(Y)$. Likewise, the integer $|X|$ is the degree 
of $S/I(X)$. The {\it index of regularity\/} of $S[u]/I(Y)$, denoted by 
${\rm reg}(S[u]/I(Y))$, is the least integer $p\geq 0$ such that
$h_Y(d)=H_Y(d)$ for $d\geq p$. Under our hypothesis, the index of
regularity of $S[u]/I(Y)$ is the 
Castelnuovo-Mumford regularity of $S[u]/I(Y)$
\cite{eisenbud-syzygies}. We will refer to ${\rm reg}(S[u]/I(Y))$
simply as the
{\it regularity\/} of $S[u]/I(Y)$. 

We shall be interested in computing the degree and the
regularity of $S[u]/I(Y)$ and $S/I(X)$ in terms of the invariants of
the clutter $\mathcal{C}$ and the number of elements of the field $K$.

Let $k\geq 2$ be an integer. A clutter is called $k$-{\it uniform\/} 
if all its edges have cardinality $k$.  

\begin{proposition}\label{maria-cafeteria-cinvestav-gen} 
Let $\mathcal{C}$ be a clutter. 
\begin{itemize}
\item[(i)] $|Y|\leq (q-1)|X|$.
\item[(ii)] If there is $A\subset V_\mathcal{C}$ so 
that $|A\cap e|=1$ for any $e\in E_\mathcal{C}$, then $|Y|=(q-1)|X|$.
\item[(iii)] If $\mathcal{C}$ is a $k$-uniform clutter and 
$\gcd(q-1,k)=1$, then $|Y|=(q-1)|X|$.
\end{itemize}
\end{proposition}

\begin{proof} (i) Let $\mathbb{T}'=\{[(z_1,\ldots,z_{s+1})]\vert\,
z_i\in K^*\}$ and 
$\mathbb{T}=\{[(z_1,\ldots,z_{s})]\vert\, z_i\in K^*\}$ be two
projective torus in $\mathbb{P}^{s}$ and 
$\mathbb{P}^{s-1}$ respectively. The projection map
$$
\mathbb{T}'\longrightarrow\mathbb{T},\ \ \ \ \ \ \
[(z_1,\ldots,z_{s+1})]\longmapsto [(z_1,\ldots,z_{s})],
$$
induces an epimorphism of multiplicative groups $\theta\colon
Y\rightarrow X$. By the fundamental homomorphism theorem for 
groups one has an isomorphism $Y/{\rm ker}(\theta)\simeq X$. Since we
have the inclusion 
$${\rm
ker}(\theta)\subset \{[(u,\ldots,u,1)]\vert\, u\in K^*\},
$$
we get $|Y|=|{\rm
ker}(\theta)||X|\leq (q-1)|X|$. This completes the proof of (i).

(ii) We may assume that $A=\{y_1,\ldots,y_\ell\}$. Let
$[(x^{v_1},\ldots,x^{v_s})]$ be 
a point in $X$ and let $\gamma$ be an arbitrary element of
$K^*$. From the equality 
\begin{eqnarray*}
&&[(x^{v_1},\ldots,x^{v_s},\gamma)]=\\ 
&&\ \ \
\left[\left(\left(\frac{x_1}{\gamma}\right)^{v_{11}}
\cdots\left(\frac{x_\ell}{\gamma}\right)^{v_{1\ell}}
x_{\ell+1}^{v_{1,\ell+1}}\cdots   
x_n^{v_{1n}},\ldots, \left(\frac{x_1}{\gamma}\right)^{v_{s1}} \cdots
\left(\frac{x_\ell}{\gamma}\right)^{v_{s\ell}}
x_{\ell+1}^{v_{s,\ell+1}}\cdots x_n^{v_{sn}},1 
\right)\right]
\end{eqnarray*}
we get that $[(x^{v_1},\ldots,x^{v_s},\gamma)]\in Y$. Let $\beta$ be
a generator of the 
cyclic group $(\mathbb{F}_q^*,\cdot)$. We can choose $P_1,\ldots,P_m$ in $X^*$
such that $X=\{[P_1],\ldots,[P_m]\}$. Hence, the set 
$$
\{[(P_1,\beta)],\ldots,[(P_1,\beta^{q-1})],\ldots,[(P_m,\beta)],\ldots,[(P_m,\beta^{q-1})]\}
$$
is contained in $Y$ and has exactly $(q-1)|X|$ elements. Therefore
$|Y|\geq (q-1)|X|$. The reverse 
inequality follows from (i).

(iii) The map $\mathbb{F}_q^*\rightarrow \mathbb{F}_q^*$, $a\mapsto
a^k$, is an isomorphism of multiplicative groups if and only if
$\gcd(q-1,k)=1$. For 
each $a\in \mathbb{F}_q^*$,
making $x_i=a$ for all $i$ in $[(x^{v_1},\ldots,x^{v_s},1)]$, we get
that the point $[(a^k,\ldots,a^k,1)]$ is in the kernel of $\theta$.
Thus $|Y|=|{\rm ker}(\theta)||X|\geq (q-1)|X|$. The reverse
inequality follows from (i). 
\end{proof} 

\begin{definition} A graph $G$ is called {\it bipartite\/} if its vertex 
set can be partitioned into two disjoint 
subsets ${V}_1$ and ${V}_2$ such 
that every edge of $G$ has one end in ${V}_1$ and one end in ${V}_2$. The pair 
$(V_1,V_2)$ is called a {\it bipartition\/} of $G$.  
\end{definition}

\begin{remark}
If $G$ is a connected bipartite graph, there is 
only one bipartition of $G$.
\end{remark}

We come to the first main result of this section.  

\begin{theorem}\label{maria-cafeteria-cinvestav} 
If $G$ is a graph, then  
$$
|Y|=\left\{
\begin{array}{ll}
{\rm(i)}\;\; (q-1)|X|,&\mbox{ if }G \mbox{ is bipartite}.\\
{\rm(ii)}\;(q-1)|X|,&\mbox{ if }\gcd(q-1,2)=1.\\
{\rm(iii)} \frac{(q-1)}{2}|X|,&\mbox{ if }G \mbox{ is not bipartite
and }\gcd(q-1,2)\neq 
1. 
\end{array}
 \right.
$$
\end{theorem}

\begin{proof} Let $(V_1,V_2)$ be a bipartition of $G$. Notice that
the set $V_1$ satisfies that $|V_1\cap e|=1$ for any $e\in E_G$.
Thus, (i) and (ii) follow from
Proposition~\ref{maria-cafeteria-cinvestav-gen}. 

(iii) We set $L=\{[(a^2,\ldots,a^2,1)]\vert\,
a\in \mathbb{F}_q^*\}\subset\mathbb{P}^s$. The projection map
$$
\{[(z_1,\ldots,z_{s+1})]\vert\, z_i\in K^*\}\rightarrow
\{[(z_1,\ldots,z_{s})]\vert\, z_i\in K^*\},\ \ \ \
[(z_1,\ldots,z_{s+1})]\mapsto [(z_1,\ldots,z_{s})], 
$$
induces an epimorphism of multiplicative groups $\theta\colon
Y\rightarrow X$. We claim that ${\rm ker}(\theta)=L$.  The
inclusion ``$\supset$'' clearly holds and is true for
any graph. 
To show the other inclusion we proceed by
contradiction. Pick $[P]\in {\rm
ker}(\theta)\setminus L$. This means that $[P]=[(b,\ldots,b,1)]$
for some $b\in \mathbb{F}_q^*$ and $b\neq a^2$ for any $a\in
\mathbb{F}_q^*$. Since $G$ is not a bipartite graph, $G$ contains and odd 
cycle $\mathcal{C}_k=\{y_1,\ldots,y_k\}$ of length $k$. We may assume
that $y^{v_1},\ldots,y^{v_k}$ are the monomials that correspond to the
edges of the cycle $\mathcal{C}_k$. Thus, any
element of $Y$ is of the form 
$$
[(x_1x_2,x_2x_3,\ldots,x_{k-1}x_k,x_1x_k,x^{v_{k+1}},\ldots,x^{v_s},1)]
$$
with $x_i\in \mathbb{F}_q^*$ for all $i$. Since the kernel of
$\theta$ is given by 
$$
{\rm ker}(\theta)=\{[(u,\ldots,u,1)]\vert\, 
u\in \mathbb{F}_q^*\}\cap
Y
$$ 
and since $[P]$ is in the kernel of $\theta$, we can write  
$$
b=x_1x_2=x_2x_3=\cdots=x_{k-1}x_k=x_1x_k
$$
for some $x_1,\ldots,x_k$ in $\mathbb{F}_q^*$. Hence 
$$
x_1=x_3=\cdots=x_k\ \mbox{ and }\ x_2=x_4=\cdots=x_{k-1}=x_1.
$$
Thus, $b=x_i^2$ for $i=1,\ldots,k$, a contradiction.  This completes
the proof of the claim. Next, we prove the equality $|L|=(q-1)/2$. 
Let $\beta$ be a generator of the
cyclic group $(\mathbb{F}_q^*,\cdot)$. In this case the image of the map 
$\mathbb{F}_q^*\rightarrow \mathbb{F}_q^*$, $a\mapsto a^2$, is a
subgroup of $\mathbb{F}_q^*$ of order $(q-1)/2$ because $\beta^2$ is a
generator of the image and this element has order $(q-1)/2$.
Therefore, $|L|=(q-1)/2$. Hence, from the isomorphism $Y/{\rm
ker}(\theta)\simeq X$ and using that $L={\rm ker}(\theta)$, 
we get $|Y|=\frac{(q-1)}{2}|X|$.
\end{proof} 

\begin{corollary}\label{maria-cafeteria-cinvestav-1} 
Let $G$ be a connected graph. Then 
$$
|Y|=\left\{
\begin{array}{ll}
{\rm(i)}\;\; (q-1)^{n-1},&\mbox{ if }G \mbox{ is bipartite}.\\
{\rm(ii)}\;(q-1)^n,&\mbox{ if }G\mbox{ is not bipartite and } q
\mbox{ is even}.\\
{\rm(iii)} \frac{(q-1)^n}{2},&\mbox{ if }G \mbox{ is not bipartite
and } q \mbox{ is odd}. 
\end{array}
 \right.
$$
\end{corollary}

\begin{proof} From \cite{algcodes}, one has that $|X|=(q-1)^{n-2}$ if
$G$ is bipartite and $|X|=(q-1)^{n-1}$ otherwise. Hence, the result
follows from Theorem~\ref{maria-cafeteria-cinvestav}. 
\end{proof}

Let $I$ be an ideal of $S$ and let $S'=S[u]$.  
By abuse of notation, we will write $I$ in place of
  $I S'$ when it is clear from context that we are using
  the generators of $I$ but extending to an ideal of the larger
  ring $S'$. 

\begin{theorem}\label{sunday-jun19-11} Let $\mathcal{C}$ be a clutter. 
\begin{itemize}
\item[(a)] If there is $A\subset V_\mathcal{C}$ so 
that $|A\cap e|=1$ for any $e\in E_\mathcal{C}$, 
then $I(Y)=I(X)+(t_1^{q-1}-t_{s+1}^{q-1})$.
\item[(b)] If $\mathcal{C}$ is a $k$-uniform clutter and
$\gcd(q-1,k)=1$, 
then $I(Y)=I(X)+(t_1^{q-1}-t_{s+1}^{q-1})$.
\end{itemize}
\end{theorem}

\begin{proof} We set $I'=I(X)+(t_1^{q-1}-t_{s+1}^{q-1})$. Notice that
$I'=I(X)+(t_i^{q-1}-t_{s+1}^{q-1})$ for any $1\leq i\leq s$. Clearly
$I'\subset I(Y)$. To show the reverse inclusion we proceed by
contradiction. Assume there is a homogeneous polynomial $f\in
I(Y)\setminus I'$. The ideal $I(Y)$ is a lattice ideal
\cite[Theorem~2.1]{algcodes}, i.e., $I(Y)$ is a binomial ideal and $t_i$ is not a
zero divisor of $S[u]/I(Y)$ for all $i$. 
Thus, we may assume that $f$ is a binomial which is a minimal
generator of $I(Y)$. Hence, we can write 
\begin{equation}\label{jun19-11}
f=t_1^{a_1}\cdots t_s^{a_s}t_{s+1}^{a_{s+1}}-t_1^{b_1}\cdots
t_s^{b_s}t_{s+1}^{b_{s+1}},
\end{equation}
such that for each $1\leq j\leq s+1$ either $a_j=0$ or $b_j=0$. We may
also assume that $a_{s+1}=0$, $b_{s+1}>0$, $a_i>0$, $b_i=0$ for some $i$. For
simplicity we assume that $i=1$. We can choose $f$ of least possible degree,
i.e., any binomial in $I(Y)$ of degree less than $\deg(f)$ belongs to
$I'$. Let $\beta$ be a generator of the cyclic group
$(\mathbb{F}_q^*,\cdot)$. 

(a) We may assume that $A=\{y_1,\ldots,y_\ell\}$. Making
$x_i=\beta^{-1}$ for $1\leq i\leq \ell$ and $x_i=1$ for $i>\ell$ in 
$[(x^{v_1},\ldots,x^{v_s},1)]$, we get 
\begin{eqnarray*}
[(x^{v_1},\ldots,x^{v_s},1)]&=&
\left[\left(\left({\beta^{-1}}\right)^{v_{11}}\cdots\left({\beta^{-1}}\right)^{v_{1\ell}}
,\ldots, \left({\beta^{-1}}\right)^{v_{s1}} \cdots
\left({\beta^{-1}}\right)^{v_{s\ell}},1 
\right)\right]\\ 
&=&[(\beta^{-1},\ldots,\beta^{-1},1)]=
[(1,\ldots,1,\beta)].
\end{eqnarray*}
Thus, $[(1,\ldots,1,\beta)]\in Y$. Then, from Eq.~(\ref{jun19-11})
and using that $f$ vanishes on 
$Y$, we get that $\beta^{b_{s+1}}=1$. Thus, $b_{s+1}=r(q-1)$
for some integer $r$. From the equality
\begin{eqnarray*}
& &f-t_{s+1}^{(r-1)(q-1)}t_1^{b_1}\cdots
t_s^{b_s}(t_1^{q-1}-t_{s+1}^{q-1})\\
& &\ \ \ \ \ \ \ \ \ \ \ \ \ \ \ \ \ \ \ \ \ \ \ \ 
=(t_1^{a_1}\cdots
t_s^{a_s}t_{s+1}^{a_{s+1}}-t_{s+1}^{(r-1)(q-1)}t_1^{b_1+(q-1)}t_2^{b_2}\cdots
t_s^{b_s})=t_1h,
\end{eqnarray*}
we obtain that the binomial $h$ is homogeneous, belongs to $I(Y)$, and
has degree less than $\deg(f)$. Thus, $h\in I'$. Consequently $f\in
I'$, a contradiction. 

(b) Making $x_i=\beta$ for all $i$ in $[(x^{v_1},\ldots,x^{v_s},1)]$
we obtain that $[(\beta^k,\ldots,\beta^k,1)]$ is in $Y$. Then, using
that $f$ vanishes on $Y$ together with Eq.~(\ref{jun19-11}), we get
that $\beta^{kb_{s+1}}=1$. As $k$ and $q-1$ are relatively prime, we
obtain that $b_{s+1}=r(q-1)$ for some integer $r$. Hence, we may
proceed as in (a) to derive a contradiction.
\end{proof}

\begin{corollary}\label{sunday-jun19-11-night} 
Let $G$ be a graph. If $G$ is bipartite or if $\gcd(q-1,2)=1$, 
then
$$I(Y)=I(X)+(t_1^{q-1}-t_{s+1}^{q-1}).$$
\end{corollary}

\begin{proof} If $G$ is bipartite, pick a bipartition $(V_1,V_2)$ of
$G$. Then, the set $V_1$ satisfies that $|V_1\cap e|=1$ for any $e\in
E_G$. Thus, the equality follows from Theorem~\ref{sunday-jun19-11}(a).
If $\gcd(q-1,2)=1$, the equality follows from
Theorem~\ref{sunday-jun19-11}(b) because any graph is $2$-uniform.
\end{proof}

The degree and the regularity of $S/I(X)$ can be computed using Hilbert series as
we now explain. 
The {\it Hilbert series} $F_X(t)$ of
$S/I(X)$ can be 
written as
\begin{equation*}
F_X(t):=\sum_{i=0}^{\infty}H_X(i)t^i=\sum_{i=0}^{\infty}\dim_K(S/I(X))_it^i=
\frac{h_0+h_1t+\cdots+h_rt^r}{1-t},
\end{equation*}
where $h_0,\ldots,h_r$ are positive integers (see \cite{Sta1}). This follows from the
fact that $S/I(X)$ is a Cohen-Macaulay standard algebra of 
dimension $1$ \cite{geramita-cayley-bacharach}. The number $r$ is the
regularity of $S/I(X)$ and $h_0+\cdots+h_r$ is the degree of
$S/I(X)$ (see \cite{Sta1} or \cite[Corollary~4.1.12]{monalg}). 

\begin{proposition}\label{jul5-11} Let $F_X(t)$ and $F_Y(t)$ be the Hilbert series of
$S/I(X)$ and $S[u]/I(Y)$ respectively. If
$I(Y)=I(X)+(t_1^{q-1}-t_{s+1}^{q-1})$, then 
\begin{enumerate}
\item[(a)] $F_Y(t)=F_X(t)(1+t+\cdots+t^{q-2})$,
\item[(b)] $|Y|=(q-1)|X|$,
\item[(c)] ${\rm reg}(S[u]/I(Y))=(q-2)+{\rm reg}(S/I(X))$, where
$u=t_{s+1}$.
\end{enumerate}
\end{proposition}

\begin{proof} As $I(X)$ and $I(Y)$ are lattice ideals
\cite[Theorem~2.1]{algcodes}, $t_i$ is not a zero divisor of
$S/I(X)$ (resp. $S[u]/I(Y)$) for $1\leq i\leq s$ (resp. $1\leq i\leq
s+1$). Hence, there are
exact sequences
\begin{eqnarray*}
&0\longrightarrow S[u]/I(Y)[-1]\stackrel{u}{\longrightarrow}
S[u]/I(Y)\longrightarrow
S[u]/(u,I(Y))\longrightarrow 0,&\\
&\ \ \ \ \ \ \ \ \ \ \ \ \ \ \ 0\longrightarrow
S/I(X)[-(q-1)]\stackrel{t_1^{q-1}}{\longrightarrow} 
S/I(X)\longrightarrow S/(t_1^{q-1},I(X))\longrightarrow 0.&
\end{eqnarray*}
Therefore, using that $S[u]/(u,I(Y))=S[u]/(u,t_1^{q-1},I(X))\simeq
S/(t_1^{q-1},I(X))$, we get
\begin{eqnarray*}
&F_Y(t)=tF_Y(t)+F(t)\ \mbox{ and }\ F_X(t)=t^{q-1}F_X(t)+F(t),& 
\end{eqnarray*}
where $F(t)$ is the Hilbert series of $S/(t_1^{q-1},I(X))$. Part (a) 
follows readily form these two equations. Recall that
$S[u]/I(Y)$ is also a Cohen-Macaulay standard algebra of dimension $1$
\cite{geramita-cayley-bacharach}. Therefore, there are unique
polynomials $g_Y(t)$ and $g_X(t)$ in $\mathbb{Z}[t]$ such that 
\begin{equation*}
F_Y(t)=g_Y(t)/(1-t)\ \mbox{ and }\ F_X(t)=g_X(t)/(1-t).
\end{equation*}
Hence, from (a), we get 
\begin{equation}\label{aug8-11}
g_Y(t)=g_X(t)(1+t+\cdots+t^{q-2}).
\end{equation}
Making $t=1$ in Eq.~(\ref{aug8-11}), we obtain 
$$
|Y|=g_Y(1)=(q-1)g_X(1)=(q-1)|X|.
$$
This proves (b). Part (c) follows from Eq.~(\ref{aug8-11}) because ${\rm
reg}(S[u]/I(Y))$ is the degree of the polynomial 
$g_Y(t)$ and ${\rm reg}(S/I(X))$ is the degree of the polynomial $g_X(t)$. 
\end{proof}

\begin{lemma}\label{maria-project} Let $X\subset\mathbb{P}^{s-1}$ and
$X'\subset\mathbb{P}^{s'-1}$ be algebraic toric sets  parameterized by
$y^{v_1},\ldots,y^{v_s}$  and $y^{v_1},\ldots,y^{v_{s'}}$  respectively.
If $s\leq s'$ and $|X|=|X'|$, then ${\rm reg}\, S'/I(X')\leq {\rm
reg}\, S/I(X)$, where $S'=K[t_1,\ldots,t_{s'}]$.
\end{lemma}

\begin{proof} Using that $I(X)$ and $I(X')$
are vanishing ideals generated by homogeneous polynomials, it is not
hard to show that $S\cap I(X')=I(X)$. Hence, we have an inclusion of
graded modules: 
$$
S/I(X)\hookrightarrow
S'/I(X').
$$
Thus, $H_{X}(d)\leq H_{X'}(d)$ for $d\geq 0$. 
Recall that $H_{X}(d)=|X|$ and $H_{X'}(d)=|X'|$ 
for $d\gg 0$ \cite{harris}. Therefore, taking 
into account that $|X|=|X'|$, we obtain: 
$$
{\rm reg}\, S'/I(X')\leq {\rm reg}\, S/I(X),
$$
as required.  
\end{proof}

\begin{remark} As $S/I(X)$  and $S'/I(X')$ are Cohen-Macaulay rings of
the same dimension and of the same degree (multiplicity),  Lemma~\ref{maria-project} can
also be shown using \cite[Proposition~3.1]{BVV}.
\end{remark}

\begin{definition}\label{projectivetorus-def} The algebraic toric set 
$\mathbb{T}=\{[(x_1,\ldots,x_s)]\in\mathbb{P}^{s-1}\vert\, x_i\in
K^*\mbox{ for all }i\}$ is called a {\it projective torus} 
in $\mathbb{P}^{s-1}$. 
\end{definition}

\begin{proposition}{\rm\cite[Theorem~1, Lemma~1]{GRH}}\label{ci-summary} If 
$\mathbb{T}$ is a projective torus 
in $\mathbb{P}^{s-1}$, then
\begin{itemize}
\item[(a)]
$I(\mathbb{T})=(t_1^{q-1}-t_s^{q-1},t_2^{q-1}-t_s^{q-1},\ldots,t_{s-1}^{q-1}-t_s^{q-1})$. 
\item[(b)] $ F_\mathbb{T}(t)=(1-t^{q-1})^{s-1}/(1-t)^s$.  
\item[(c)] ${\rm reg}(S/I(\mathbb{T}))=(s-1)(q-2)$ and ${\rm
deg}(S/I(\mathbb{T}))=(q-1)^{s-1}$.
\end{itemize}
\end{proposition}

\begin{definition} Let $G$ be a bipartite graph with bipartition
$(V_1,V_2)$. If every vertex in $V_1$ is joined to every vertex in
$V_2$, then $G$ is called a {\it complete bipartite graph}. If $V_1$
and $V_2$ have $s_1$ and $s_2$ vertices respectively, we 
denote a complete bipartite graph by 
${\mathcal K}_{s_1,s_2}$. A {\it spanning subgraph\/} of a graph $G$ is a 
subgraph containing all the vertices of $G$. 
\end{definition}

\begin{theorem}\label{upper-lower-bounds-reg-bip} 
Let $G$ be a connected bipartite graph with
bipartition $(V_1,V_2)$ and let $X$ be the projective algebraic toric set parameterized by
the edges of $G$. If $|V_2|\leq|V_1|$, then 
$$
(|V_1|-1)(q-2)\leq{\rm reg}\, S/I(X)\leq (|V_1|+|V_2|-2)(q-2).
$$
Furthermore, equality on the left occurs if $G$ is a complete bipartite
graph and equality on the right occurs if $G$ is a tree.
\end{theorem}

\begin{proof} We set $|V_i|=s_i$ for $i=1,2$. First we prove the 
inequality on the left. Let $X_1\subset\mathbb{P}^{s_1-1}$ and
$X_2\subset\mathbb{P}^{s_2-1}$ be two projective torus and let
$X'\subset\mathbb{P}^{s_1s_2-1}$ be
the algebraic toric set parameterized by the edges of the complete
bipartite graph $\mathcal{K}_{s_1,s_2}$ with bipartition $(V_1,V_2)$.
According to \cite{GR} the corresponding Hilbert functions are related by 
$$
H_{X'}(d)=H_{X_1}(d)H_{X_2}(d)\ \mbox{ for }\ d\geq 0.
$$ 
By Proposition~\ref{ci-summary}(c) the regularity index of 
$K[t_1,\ldots,t_{s_i}]/I(X_i)$ is equal to $(s_i-1)(q-2)$.
Hence, the regularity index of $K[t_1,\ldots,t_{s_1s_2}]/I(X')$ is
equal to $(s_1-1)(q-2)$. Therefore, taking into account that $|X|=|X'|=(q-1)^{s_1+s_2-2}$
\cite{algcodes}, by Lemma~\ref{maria-project} we obtain: 
$$
(s_1-1)(q-2)={\rm reg}\, K[t_1,\ldots,t_{s_1s_2}]/I(X')\leq {\rm
reg}\, S/I(X),
$$
as required. 

Next, we prove the inequality on the right. 
Let $H$ be an spanning tree of $G$, that is, $H$ is a subgraph of $G$
such that $H$ is a tree that contains every vertex of $G$. Consider
the projective algebraic toric set $X_3$ parameterized by the edges of $H$. We
may assume that $v_1,\ldots,v_{s_1+s_2-1}$ are the characteristic
vectors of the edges of $H$. As $H$ is a tree, by
\cite[Corollary~3.8]{algcodes}, one has
$|X_3|=(q-1)^{s_1+s_2-2}$. Since
$X_3$ is contained in a projective torus $\mathbb{T}'$ in
$\mathbb{P}^{s_1+s_2-2}$ and since $|\mathbb{T}'|=(q-1)^{s_1+s_2-2}$, 
we get that $X_3=\mathbb{T}'$, that is, $X_3$ is a projective torus 
in $\mathbb{P}^{s_1+s_2-2}$. Therefore, by
Proposition~\ref{ci-summary}(c), we obtain
\begin{equation}\label{jun10-11}
{\rm reg}\, K[t_1,\ldots,t_{s_1+s_2-1}]/I(X_3)=(s_1+s_2-2)(q-2).
\end{equation}
Using
\cite[Corollary~3.8]{algcodes}, we get that $|X|$ and $|X_3|$ are 
equal to $(q-1)^{s_1+s_2-2}$. Then, by Lemma~\ref{maria-project} and
Eq.~(\ref{jun10-11}), we
get
$$
{\rm reg}\, S/I(X)\leq {\rm reg}\,
K[t_1,\ldots,t_{s_1+s_2-1}]/I(X_3)=(s_1+s_2-2)(q-2),
$$
as required. 
\end{proof}

A connected graph is always a spanning subgraph of a complete graph. 
An interesting open problem is to compute the regularity of $S/I(X)$
for a complete graph because---using Lemma~\ref{maria-project} and
\cite[Corollary 3.8]{algcodes}---this would
give an optimal lower bound for the regularity
of any connected non-bipartite graph (see the proof of
Theorem~\ref{upper-lower-bounds-reg-bip}). 

For even cycles the regularity of $S/I(X)$ and the basic parameters of
parameterized codes over even cycles were studied in \cite{even-cycles}. 

\begin{corollary}{\cite[Corollary~3.1]{even-cycles}}\label{jun11-11-1} 
 If $G$ is an even cycle of length $2k$, then 
$${\rm reg}\, S/I(X)\geq (k-1)(q-2).$$
\end{corollary}

\begin{proof} If $(V_1,V_2)$ is the bipartition of $G$, then
$|V_1|=|V_2|=k$. Hence, the inequality follows from
Theorem~\ref{upper-lower-bounds-reg-bip}. 
\end{proof}

The reverse inequality is also true but it is much harder to prove.

\begin{theorem}{\cite{neves-vaz-pinto-preprint}}
\label{neves-vaz-pinto}\label{jun11-11-2}
If $G$ is an even cycle of length $2k$, then 
$${\rm reg}(S/I(X))\leq (k-1)(q-2).$$
\end{theorem}

A cycle containing all the vertices of a graph is called a {\it
Hamilton cycle}. A graph containing a Hamilton cycle is
called {\it Hamiltonian}.   

\begin{corollary}\label{jun11-11-3} 
If $G$ is a Hamiltonian bipartite graph with $2k$
vertices, then 
$${\rm reg}(S/I(X))=(k-1)(q-2).$$
\end{corollary}

\begin{proof} Let $(V_1,V_2)$ be the bipartition of $G$, let $H$ be
a Hamilton cycle of $G$, and let $\mathcal{K}_{k,k}$ be the complete
bipartite graph with bipartition $(V_1,V_2)$. Notice 
that $H$ is a spanning subgraph of $G$ and $G$ is a spanning subgraph of
$\mathcal{K}_{k,k}$. Therefore, applying Lemma~\ref{maria-project}
together with Theorem~\ref{upper-lower-bounds-reg-bip} and
Theorem~\ref{jun11-11-2}, the equality follows.
\end{proof}

The next open problem is known as the Eisenbud-Goto regularity
conjecture \cite{eisenbud-goto}.

\begin{conjecture} If $\mathfrak{p}\subset
(t_1,\ldots,t_s)^2$ is a prime graded ideal of $S$, then 
$${\rm
reg}(S/\mathfrak{p})\leq {\rm
deg}(S/\mathfrak{p})-{\rm codim}(S/\mathfrak{p}).
$$
\end{conjecture}

There is a version of this conjecture, for square-free monomial
ideals whose Stanley-Reisner complex is connected in codimension $1$, that has been shown in
\cite{terai}. We will show the Eisenbud-Goto regularity conjecture for
vanishing ideals over Hamiltonian bipartite graphs.

\begin{lemma}\label{aug14-11} Let $k\geq 2$ and $q\geq 3$ be two
integers. Then {\rm (i)} $2^{2k-2}\geq (k-1)(k+2)$, and {\rm (ii)} 
$(q-1)^{2k-2}\geq (k-1)(q+k-1).$
\end{lemma}

\begin{proof} The inequality in (i) follows readily by induction on
$k$. The inequality in (ii) follows by induction on $q$ and using (i).
\end{proof}

\begin{corollary}\label{eisenbud-goto-hamiltonian} If $G$ is a Hamiltonian 
bipartite graph, then   
$${\rm
reg}(S/I(X))\leq {\rm
deg}(S/I(X))-{\rm codim}(S/I(X)).
$$
\end{corollary}

\begin{proof} The graph $G$ has $s$ edges and $n$ vertices. Since $G$
is Hamiltonian and bipartite, $n=2k$ for some integer $k\geq 2$ and $G$ has a
bipartition $(V_1,V_2)$ with $|V_i|=k$ for $i=1,2$. 
Thus, $s\leq k^2$. Hence, by Lemma~\ref{aug14-11}, we have:
$$
(s-1)+(q-2)(k-1)\leq (k^2-1)+(q-2)(k-1)=(k-1)(q+k-1)\leq (q-1)^{2k-2}.
$$
To complete the proof notice that ${\rm deg}\, S/I(X)$ is 
$|X|=(q-1)^{2k-2}$ \cite[Corollary~3.8]{algcodes}, ${\rm codim}\,
S/I(X)$ is $s-1$ \cite{geramita-cayley-bacharach} and ${\rm reg}\,
S/I(X)$ is $(q-2)(k-1)$ (see Corollary~\ref{jun11-11-3}).
\end{proof}

\begin{definition} Let $\mathcal{C}$ be a clutter and let $y_i$ be a
vertex. We say $y_i$ is a {\it free
vertex} of ${\mathcal C}$ if $y_i$ only appears in one of the edges
of $\mathcal C$. 
\end{definition} 

\begin{definition}\label{support-of-a-vector} If $a\in {\mathbb R}^n$, its {\it
support\/} is defined as ${\rm supp}(a)=\{i\, |\, a_i\neq 0\}$. The
{\it support} of the monomial $y^a$ is defined as 
${\rm supp}(y^a)=\{y_i\, |\, a_i\neq 0\}$.  
\end{definition}

\begin{theorem}\label{june-2011--dictamina} Let $\mathcal{C}$ be a clutter and let $X'$ be the
projective algebraic toric set parameterized by $y^{v_1},\ldots,y^{v_{s-1}}$. If
$y_n$ is a free vertex of $\mathcal{C}$ and $y_n\in{\rm
supp}(y^{v_s})$, then 
\begin{itemize}
\item[(a)] $I(X)=I(X')+(t_1^{q-1}-t_s^{q-1})$.
\item[(b)] ${\rm reg}\, S/I(X)={\rm reg}\, S'/I(X')+(q-2)$, where
$S'=K[t_1,\ldots,t_{s-1}]$. 
\item[(c)] ${\rm deg}\, S/I(X)=(q-1){\rm deg}\, S'/I(X')$.
\end{itemize}
\end{theorem}

\begin{proof} (a) We set $I'=I(X')+(t_1^{q-1}-t_s^{q-1})$. Clearly
$I'\subset I(X)$. Recall that $I(X)$ is generated by a finite set of
binomials \cite{algcodes}. To show the inclusion $I(X)\subset I'$ we
proceed by contradiction. Pick a homogeneous binomial $g$ in $I(X)\setminus I'$ of least
possible degree, i.e., any binomial of $I(X)$ of degree less than
$\deg(g)$ belongs to $I'$. We can write 
\begin{equation*}
g=t_1^{a_1}\cdots t_s^{a_s}-t_1^{b_1}\cdots
t_s^{b_s},
\end{equation*}
with ${\rm supp}(t_1^{a_1}\cdots t_s^{a_s})\cap{\rm
supp}(t_1^{b_1}\cdots t_s^{b_s})=\emptyset$. 
If $a_s=b_s=0$, then $g\in I(X')$ which is impossible. Thus, we may
assume that $a_s>0$ and $b_s=0$. Thus, $b_i>0$ for some $1\leq i\leq s-1$. For simplicity of
notation, we assume that $i=1$. Making $x_i=1$ for $i=1,\ldots,n-1$ in
the equality 
$$
(x^{v_1})^{a_1}\cdots (x^{v_s})^{a_s}=(x^{v_1})^{b_1}\cdots
(x^{v_s})^{b_s}
$$
we get $x_n^{a_s}=1$ for any $x_n\in K^*$. In particular, if $\beta$ is
a generator of the cyclic group $(K^*,\cdot)$ and $x_n=\beta$, we get
$\beta^{a_s}=1$. Hence, we can write $a_s=\mu(q-1)$ for some integer
$\mu$. As $b_1>0$, one has the equality 
\begin{eqnarray*}
h&=&(t_1^{a_1}\cdots t_s^{a_s}-t_1^{b_1}\cdots
t_s^{b_s})+(t_1^{q-1}-t_s^{q-1})(t_s^{(\mu-1)(q-1)}t_1^{a_1}\cdots
t_{s-1}^{a_{s-1}})\nonumber\\ &=&-t_1^{b_1}\cdots
t_s^{b_s}+t_1^{q-1}t_s^{(\mu-1)(q-1)}t_1^{a_1}\cdots
t_{s-1}^{a_{s-1}}=t_1g_1
\end{eqnarray*}
for some binomial $g_1$. Notice that $h\neq 0$, otherwise $g\in I'$
which is impossible. Therefore $g_1$ is in $I(X)\setminus I'$ and has degree less
than $\deg(g)$, a contradiction to the choice of $g$.  

(b)  We set $B=S'/(I(X'),t_1^{q-1})$. By part (a), $B=S/(I(X),t_s)$.
There are exact sequences

\begin{eqnarray*}
&0\longrightarrow S/I(X)[-1]\stackrel{t_s}{\longrightarrow}
S/I(X)\longrightarrow
B\longrightarrow 0,&\ \ \ \ \ \ \ \ \ \ \ \ \ \ \ \ \ \ \ \ \ \ \ \ \\
 &\ \ \ \ \ \ \ \ \ \ \ \ \ \ \ \ \ \ \ \ \ \ \ \ 
0\longrightarrow S'/I(X')[-(q-1)]\stackrel{t_1^{q-1}}{\longrightarrow}
S'/I(X')\longrightarrow B\longrightarrow 0.&
\end{eqnarray*}
Therefore 
\begin{eqnarray*}
&(1-t)F_X(t)=F(B,t)\ \mbox{ and }\ (1-t^{q-1})F_{X'}(t)=F(B,t)& 
\end{eqnarray*}
where $F(B,t)$ is the Hilbert series of $B$. Thus,
$F_X(t)=(1+t+\cdots+t^{q-2})F_{X'}(t)$. From this equality (b)
follows (see the last part of the proof of
Proposition~\ref{jul5-11}). 
Part (c) also follows from this equality.
\end{proof}

A graph with exactly one cycle is called {\it unicyclic\/}.

\begin{corollary}\label{jul31-11} Let $G$ be a connected bipartite graph. If $G$ is
unicyclic with a cycle of length $2k$, then 
${\rm reg}\, S/I(X)=(q-2)(n-k-1)$. 
\end{corollary}

\begin{proof} We proceed by induction on $n$. Notice that $s=n$, i.e., the number of
edges of $G$ is equal to the number of vertices of $G$. This follows
using that $G$ is connected and unicyclic. 
If $G$ is a
cycle, then $s=2k$ and the result follows from
Corollary~\ref{jun11-11-3}. If $G$ is not a cycle, then $G$ has a
vertex $y_i$ of degree $1$. We may assume that $y^{v_s}$ is the only
monomial that contains $y_i$. Hence, by Theorem~\ref{june-2011--dictamina}, we get 
$$
{\rm reg}\, S/I(X)={\rm reg}\, S'/I(X')+(q-2), 
$$
where $X'$ is parameterized by $y^{v_1},\ldots,y^{v_{s-1}}$ and 
$S'=K[t_1,\ldots,t_{s-1}]$. To complete the proof notice that by
induction hypothesis we have ${\rm reg}\, S'/I(X')=(n-2-k)(q-2)$.
\end{proof}

A graph ${G}$ is {\it chordal\/} if every
cycle of ${G}$ of length $n\geq 4$ has a chord. A {\it chord} of a
cycle is an edge joining two non adjacent vertices of the cycle. If
$v$ is a vertex of a graph $G$, then
its {\it neighbor set},  
denoted by $N_G(v)$, is
the set of vertices of $G$ adjacent to $v$. Let $U$ be a set of
vertices of $G$.  The {\it induced
subgraph\/} $G[U]$ is the maximal subgraph of $G$ with
vertex set $U$.  

\begin{lemma}{\rm\cite[Theorem~8.3]{Toft}}\label{toft-lemma}
Let $G$ be a chordal graph and let ${\mathcal K}$ be a complete 
subgraph of $G$. If ${\mathcal K}\neq G$, then there is $v\not\in
V_{\mathcal K}$ 
such that $G[N_G(v)]$ is a complete subgraph.
\end{lemma}

Let $G$ be a graph. A {\it clique\/} of $G$ is a set of mutually adjacent 
vertices. The {\it clique clutter\/} of $G$, denoted
by ${\rm cl}(G)$, is the clutter on $V_G$ whose edges are the 
maximal cliques of $G$ (maximal with respect to inclusion). 
Given $v\in V_G$, by $G \setminus\{v\}$, we mean
the graph formed from $G$ by deleting $v$, and all
edges incident to $v$.   

\begin{corollary} Let $\mathcal{C}={\rm cl}(G)$ be the clique clutter of a chordal
graph $G$. If $\mathcal{C}$ has $s$ edges, 
then ${\rm reg}\, S/I(X)=(q-2)(s-1)$ and ${\rm deg}\,
S/I(X)=(q-1)^{s-1}$.
\end{corollary}

\begin{proof} From Lemma~\ref{toft-lemma}, the clique clutter of $G$
has a free vertex. We denote this vertex by $y_n$. Using that
$G\setminus\{y_n\}$ is a chordal graph, together with
Theorem~\ref{june-2011--dictamina}, the result follows by induction 
on the number of edges of ${\rm cl}(G)$. 
\end{proof}

\begin{corollary}\label{jul5-11-1} If $G$ is a connected bipartite
graph with a largest cycle of length $2k$, then 
$$
{\rm reg}\, S/I(X)\leq (q-2)(n-k-1).
$$
\end{corollary}

\begin{proof} Let $C$ be a cycle of $G$ of length $2k$. It is not
hard to see, by induction on the number of vertices, that $G$ has a
unicyclic connected subgraph $H$ with $V_G=V_H$ and whose only cycle is
$C$. Let $X'$ be the set parameterized by the edges of $H$. We set
$S'=K[t_1,\ldots,t_n]$, where $t_1,\ldots,t_n$ are the variables that
correspond to the monomials defining the edges of $H$. By
Corollary~\ref{jul31-11}, ${\rm reg}(S'/I(X'))$ is equal to 
$(q-2)(n-k-1)$. Notice that $|X|=|X'|=(q-1)^{n-2}$ because $H$ and $G$
are both connected bipartite graphs with $n$ vertices (see
\cite[Corollary~3.8]{algcodes}). 
Thus, by 
Lemma~\ref{maria-project}, ${\rm reg}(S/I(X))\leq {\rm
reg}(S'/I(X'))$. This proves the required inequality. 
\end{proof}

\begin{example} Let $K=\mathbb{F}_3$ be the field with $3$ elements and let
$X$ be the projective algebraic toric set parameterized by the monomials:
$$
x_1x_6,\, 
x_1x_2,\, 
x_1x_8,\, 
x_3x_2,\, 
x_3x_4,\, 
x_5x_6,\, 
x_5x_4,\,
x_5x_8,\, 
x_7x_2,\, 
x_7x_4. 
$$
The graph $G$, whose edges correspond to these monomials, is connected
and bipartite with bipartition $V_1=\{x_1,x_3,x_5,x_7\}$,
$V_2=\{x_6,x_2,x_4,x_8\}$. All vertices of this graph have degree at
least two. The largest cycle of $G$ has length $6$. Thus, by a direct application
of Corollary~\ref{jul5-11-1} and
Theorem~\ref{upper-lower-bounds-reg-bip}, we get $3\leq {\rm
reg}(S/I(X))\leq 4$. Using {\em Macaulay\/}$2$ \cite{mac2} it is seen
that the regularity of $S/I(X)$ is equal to $4$. 
\end{example}

\section{Applications to coding theory}\label{applications}

We continue to use the notation and definitions used in
Sections~\ref{intro-rs-codes} and \ref{degree-and-reg}. In this
section we recall the well known interconnections between 
the algebraic invariants of vanishing ideals and the basic parameters of affine and
projective parameterized linear codes. Then we present upper and lower bounds for
the minimum distance of parameterized codes arising from connected
bipartite graphs. 

Some families of evaluation codes have been studied extensively using
commutative algebra methods and especially Hilbert functions, 
see \cite{delsarte-goethals-macwilliams,
duursma-renteria-tapia,gold-little-schenck,GRT,
algcodes,sorensen}. In this
section we use these methods to study parameterized codes over 
finite fields. 

Let $S=K[t_1,\ldots,t_s]=\oplus_{d=0}^\infty S_d$ 
be a polynomial ring 
over the field $K$ with the standard grading, let $Q_1,\ldots,Q_r$
be the points of $X^*$, and let $S_{\leq d}$ be the set of polynomials of $S$ of
degree at most $d$. 
\begin{definition} The {\it evaluation map\/} ${\rm ev}_d\colon
S_{\leq d}\rightarrow K^{|X^*|}$, $f\mapsto
\left(f(Q_1),\ldots,f(Q_r)\right)$, 
defines a $K$-linear map. The image of ${\rm ev}_d$, denoted by $C_{X^*}(d)$,
is a {\it linear code} which is called a {\it parameterized affine code\/} of
degree $d$ on $X^*$, by a {\it linear code\/} we mean a linear subspace of
$K^{|X^*|}$. 
\end{definition}

Parameterized affine codes are special types of
Reed-Muller codes (in the sense of \cite[p.~37]{tsfasman}) and of evaluation 
codes
\cite{duursma-renteria-tapia,gold-little-schenck,GRT,lachaud}. 
If $s=n=1$ and $v_1=1$, then
$X^*=\mathbb{F}_q^*$ and we obtain the classical Reed-Solomon code of
degree $d$ \cite[p.~33]{tsfasman}.    

The {\it dimension\/} and the {\it length\/} of $C_{X^*}(d)$ 
are given by $\dim_K C_{X^*}(d)$ and $|{X^*}|$ respectively. The dimension
and the length 
are two of the {\it basic parameters} of a linear code. A third
basic parameter is the {\it minimum
distance\/} which is given by 
$$\delta_{X^*}(d)=\min\{\|v\|
\colon 0\neq v\in C_{X^*}(d)\},$$ 
where $\|v\|$ is the number of non-zero
entries of $v$. The basic parameters of $C_{X^*}(d)$ are related by the
{\it Singleton bound\/} for the minimum distance:
$$
\delta_{X^*}(d)\leq |{X^*}|-\dim_KC_{X^*}(d)+1.
$$

Two of the parameters of $C_{X^*}(d)$ can be expressed 
using Hilbert functions of standard graded algebras as is seen below.

\begin{definition}\label{def-param-proj-code}
The evaluation map
$$
{\rm ev}'_d\colon S[u]_d\rightarrow K^{|Y|},\ \ \ \ \ 
f\mapsto
\left(\frac{f(Q_1,1)}{f_0(Q_1,1)},\ldots,\frac{f(Q_r,1)}{f_0(Q_r,1)}\right),
$$
where $f_0(t_1,\ldots,t_{s+1})=t_1^d$, defines a linear map of 
$K$-vector spaces. The image of ${\rm ev}'_d$, denoted by $C_Y(d)$, is
called a {\it parameterized
projective code\/} of
degree $d$ on the set $Y$. The minimum distance of $C_Y(d)$ is denoted
by $\delta_Y(d)$.
\end{definition}

\begin{definition}
The {\it affine Hilbert function\/} of $S/I(X^*)$
is given by 
$$
H_{X^*}(d):=\dim_K\, S_{\leq d}/I(X^*)_{\leq d}
$$
where $I(X^*)_{\leq d}=S_{\leq d}\cap I(X^*)$. 
\end{definition}

This paper is motivated by the fact that the degree and the Hilbert
function of $S[u]/I(Y)$ are related to the basic parameters of 
parameterized affine linear codes: 

\begin{theorem}{\cite[Theorem~2.4]{affine-codes}}\label{bridge-affine-projective}
{\rm(a)} $C_{X^*}(d)\simeq C_Y(d)$ as $K$-vector spaces. 

{\rm(b)} The parameterized codes $C_{X^*}(d)$ and $C_Y(d)$ have the
same parameters. 

{\rm(c)} The dimension and the length of $C_{X^*}(d)$ are 
$H_Y(d)$ and ${\rm deg}(S[u]/I(Y))$ respectively.

{\rm(d)} \cite[Remark~2.5]{affine-codes}
$H_Y(d)=H_{X^*}(d)$ for $d\geq 0$ $($cf.
\cite[Remark~5.3.16]{singular}$)$.  
\end{theorem}

\begin{lemma}\label{comparing-md} Let $X\subset\mathbb{P}^{s-1}$ and
$X'\subset\mathbb{P}^{s'-1}$ be algebraic toric sets  parameterized by
$y^{v_1},\ldots,y^{v_s}$  and $y^{v_1},\ldots,y^{v_{s'}}$  respectively.
If $s\leq s'$ and $|X|=|X'|$, then $\delta_{X'}(d)\leq \delta_X(d)$. 
\end{lemma}

\begin{proof} We can choose $P_1,\ldots,P_m$
in $X^*$ so that $X=\{[P_1],\ldots,[P_m]\}$. There is a well defined epimorphism 
$$
\phi\colon X'\rightarrow X,\ \ \ \ \ \ \ \
[(x^{v_1},\ldots,x^{v_{s'}})]\mapsto [(x^{v_1},\ldots,x^{v_s})], 
$$
induced by the map $[(\alpha_1,\ldots,\alpha_{s'})]\mapsto[(\alpha_1,\ldots,\alpha_s)]$. By
hypothesis $|X'|=|X|$. Hence, the map $\phi$ 
is an isomorphism of multiplicative groups. Thus,
we can write $X'=\{[P_1'],\ldots,[P_m']\}$ so that $[P_i']$ maps under
$\phi$ to $[P_i]$ for all $i$. 
Pick $F$ in $S_d=K[t_1,\ldots,t_s]_d$ such that
$\delta_X(d)=|\{P_i\vert\, F(P_i)\neq 0\}|$. Notice that $F$ is also a
polynomial in $S[t_{s+1},\ldots,t_{s'}]_d$. Since $F(P_i)\neq 0$ if and
only if $F(P_i')\neq 0$, we get $\delta_{X'}(d)\leq \delta_X(d)$. 
\end{proof}

\begin{theorem}\label{jun25-11} Let $G$ be a connected bipartite graph with
bipartition $(V_1,V_2)$. Then 
$$
\delta_{X_1}(d)\delta_{X_2}(d)\leq \delta_{X}(d)\leq \delta_{X_3}(d),\
\mbox{ for }\ d\geq 1,
$$
where $X_3$ is a projective torus in $\mathbb{P}^{|V_1|+|V_2|-2}$ and
$X_i$ is a projective torus in $\mathbb{P}^{|V_i|-1}$ for $i=1,2$.
\end{theorem}

\begin{proof} We set $|V_i|=s_i$ for $i=1,2$. First we prove the 
inequality on the left. Let $X'\subset\mathbb{P}^{s_1s_2-1}$ be
the projective algebraic toric set parameterized by the edges of the complete
bipartite graph $\mathcal{K}_{s_1,s_2}$ with bipartition $(V_1,V_2)$.
According to \cite{GR} the minimum distances are related by  
$\delta_{X'}(d)=\delta_{X_1}(d)\delta_{X_2}(d)$ for $d\geq 1$. Recall
that $|X|=|X'|=(q-1)^{s_1+s_2-2}$ \cite{algcodes}. Therefore, by 
Lemma~\ref{comparing-md}, we obtain the inequality on the left. 

Let $H$ be an spanning tree of $G$, that is, $H$ is a subgraph of $G$
such that $H$ is a tree that contains every vertex of $G$. Consider
the algebraic toric set $X''$ parameterized by the edges of $H$. We
may assume that $v_1,\ldots,v_{s_1+s_2-1}$ are the characteristic
vectors of the edges of $H$. Notice that the set $X''$ is
a projective torus in $\mathbb{P}^{s_1+s_2-2}$, i.e., $X''=X_3$ (see
the proof of Theorem~\ref{upper-lower-bounds-reg-bip}). Using
\cite[Corollary~3.8]{algcodes}, we get that $|X|$ and $|X_3|$ are 
equal to $(q-1)^{s_1+s_2-2}$. Therefore, by 
Lemma~\ref{comparing-md}, we obtain the inequality on the right. 
\end{proof}

Lower bounds for
the minimum distance 
of evaluation codes have been
shown when $X$ is any complete intersection reduced set of points in
a projective space
\cite{ballico-fontanari,gold-little-schenck,hansen}, and when $X$ is a
reduced Gorenstein set of points \cite{tohaneanu}. Upper bounds for
the minimum distance of certain parameterized codes are given in
\cite{algcodes,d-compl}.

There is a nice recent formula for the minimum distance of a
parameterized code over a projective torus.

\begin{theorem}{\cite[Theorem~3.4]{ci-codes}}\label{maria-vila-hiram-eliseo} 
If $\mathbb{T}$ is a projective torus in $\mathbb{P}^{n-1}$ and $d\geq 1$, 
 then the minimum distance of $C_\mathbb{T}(d)$ is given by 
$$
\delta_{\mathbb{T}}(d)=\left\{\begin{array}{cll}
(q-1)^{n-(k+2)}(q-1-\ell)&\mbox{if}&d\leq (q-2)(n-1)-1,\\
1&\mbox{if}&d\geq (q-2)(n-1),
\end{array}
 \right.
$$
where $k$ and $\ell$ are the unique integers such that $k\geq 0$,
$1\leq \ell\leq q-2$ and $d=k(q-2)+\ell$. 
\end{theorem}

\begin{corollary}\label{upper-lower-bounds-md} 
Let $G$ be a connected bipartite graph with
bipartition $(V_1,V_2)$. Then 
$$
(q-2)^2(q-1)^{|V_1|+|V_2|-4}\leq \delta_{X}(1)\leq
(q-2)(q-1)^{|V_1|+|V_2|-3}.
$$
\end{corollary}

\begin{proof} It follows readily from Theorems~\ref{jun25-11} and
\ref{maria-vila-hiram-eliseo}. 
\end{proof}

\begin{example} Let $G$ be a cycle of length $6$. If $K=\mathbb{F}_5$,
then $144\leq \delta_{X}(1)\leq 192$. The exact value of $\delta_{X}(1)$ is $186$.
\end{example}

\begin{proposition}\label{jun12-11}
Let $\mathcal{C}$ be a clutter and let $G$ be a graph. {\rm(a)} If
there is $A\subset V_\mathcal{C}$ so  
that $|A\cap e|=1$ for any $e\in E_\mathcal{C}$, then 
$\delta_Y(d)\leq(q-1)\delta_X(d)$ for any $d\geq 1$. {\rm (b)} If $G$
is a connected bipartite graph,  
then $\delta_Y(1)=(q-1)\delta_X(1)$. 
\end{proposition}

\begin{proof} We may assume that $A=\{y_1,\ldots,y_\ell\}$. Let
$\beta$ be a generator 
of the cyclic group
$(\mathbb{F}_q^*,\cdot)$. We can choose $P_1,\ldots,P_m$
in $X^*$ so that $X=\{[P_1],\ldots,[P_m]\}$. If $P=P_\ell$ for some
$\ell$, then we can write
$P=(x^{v_1},\ldots,x^{v_s})$. We set $\gamma_i=\beta^i$. 
From the equality
\begin{eqnarray*}
\gamma_i P&=&\left(({\gamma_i}{x_1})^{v_{11}}\cdots({\gamma_i}{x_\ell})^{v_{1\ell}}
x_{\ell+1}^{v_{1,\ell+1}}\cdots   
x_n^{v_{1n}},\ldots,({\gamma_i}{x_1})^{v_{s1}} \cdots
({\gamma_i}{x_\ell})^{v_{s\ell}}
x_{\ell+1}^{v_{s,\ell+1}}\cdots x_n^{v_{sn}}\right)
\end{eqnarray*}
we get that $\gamma_i P\in X^*$. By 
Proposition~\ref{maria-cafeteria-cinvestav-gen}, we have that
$|Y|=(q-1)|X|$. Therefore
$$
Y=\{[(\beta P_1,1)],\ldots,[(\beta^{q-1}P_1,1)],\ldots,[(\beta P_m,1)],
\ldots,[(\beta^{q-1}P_m,1)]\}.
$$
To show (a) pick $0\neq f\in S_d$ such that $\delta_X(d)=|\{P_i\vert\,
f(P_i)\neq 0\}|$. Notice that $f(P_i)=0$ if and only if
$f(\beta^jP_i)=0$ for $1\leq j\leq q-1$. Hence, $f$ does not vanish in
exactly $(q-1)\delta_X(d)$ points of $Y$. Consequently $\delta_Y(d)$
is at most $(q-1)\delta_X(d)$.

To show (b) pick a polynomial $F$ in $S[u]_1$ such that
$\delta_Y(1)=|\{Q\in Y\vert\, F(Q)\neq 0\}|$. If $F\in S$, then 
$F(P_i)\neq 0$ if and only if $F(\beta^jP_i,1)\neq 0$ for some $1\leq
j\leq q-1$ if and only if $F(\beta^jP_i,1)\neq 0$ for all $1\leq
j\leq q-1$. Hence, $\delta_Y(1)=(q-1)r_0$, where $r_0=|\{P_i\vert\,
F(P_i)\neq 0\}|$. Thus, $(q-1)\delta_X(1)\leq (q-1)r=\delta_Y(1)$.
Then, using (a), we get $\delta_Y(1)=(q-1)\delta_X(1)$. We
may now assume that $F=\lambda_1t_1+\cdots+\lambda_st_s+u$, 
where $\lambda_i\in K$ for all $i$. It is not hard to verify that 
if $F(\beta^\ell P_i,1)=0$ for some $1\leq \ell\leq q-1$, then
$F(\beta^j P_i,1)\neq 0$ for all $1\leq j\leq q-1$, $j\neq \ell$.
Hence, the number of zeros in $Y$ of $F$ is at most $|X|=(q-1)^{n-2}$.
Consequently one has
$$
(q-1)^{n-2}(q-2)=|Y|-(q-1)^{n-2}\leq\delta_Y(1)\leq
(q-1)\delta_X(1)\leq (q-1)^{n-2}(q-2). 
$$
The first equality is shown in
Corollary~\ref{maria-cafeteria-cinvestav-1} and the 
last inequality follows from
Corollary~\ref{upper-lower-bounds-md}. Therefore, we have equality
everywhere. In particular $\delta_Y(1)=(q-1)\delta_X(1)$.
\end{proof}

\section{Complete intersection $I(Y)$ from
clutters}\label{ci-clutters-affine}

We continue to use the notation and definitions used in
Sections~\ref{intro-rs-codes} and \ref{degree-and-reg}.  
In this section we characterize when $I(Y)$ is
a complete intersection in algebraic and geometric terms. For graphs,
we describe in graph theoretical terms and in terms of the
number of elements of the base field when $I(Y)$ is a complete
intersection.  

\begin{lemma}\label{jan6-10-y}
Let $\mathcal{C}$ be a clutter. If $f\neq 0$ is a homogeneous
polynomial of $I(Y)$ of the form $t_i^b-t^c$ with $b\in\mathbb{N}$,
$c\in\mathbb{N}^s$ and $i\notin{\rm supp}(c)$, 
then $\deg(f)\geq q-1$. Moreover if $b=q-1$, then
$f=t_i^{q-1}-t_j^{q-1}$ for some $j\neq i$.
\end{lemma}

\begin{proof} It follows adapting the proof of \cite[Lemma~3.4]{d-compl}.
\end{proof}

\begin{definition}\label{jul25-11}
The ideal $I(Y)$ is called a {\it complete
intersection\/} if it can be generated by $s$ homogeneous
polynomials of $S[u]$.
\end{definition}

\begin{lemma}\label{1-jun-11} Let $\mathcal{C}$ be a clutter. If
$I(Y)$ is a complete intersection, then 
$$I(Y)=(t_1^{q-1}-t_{s+1}^{q-1},\ldots,t_s^{q-1}-t_{s+1}^{q-1}).$$ 
\end{lemma}

\begin{proof} Taking into account Lemma~\ref{jan6-10-y}, we can use
the same proof of \cite[Theorem~4.4]{ci-codes}. 
\end{proof}

\begin{definition}\label{projective-closure-def} The 
{\it projective closure\/} of $X^*$, denoted by $\overline{X^*}$, is
given by $\overline{X^*}:=\overline{Y}$, where 
$\overline{Y}$ is the closure of $Y$ in the Zariski topology of $\mathbb{P}^s$.
\end{definition}

The next theorem complements a result of \cite{ci-codes} showing that
$I(X)$ is a complete intersection 
if and only if $X$ is a projective torus. 

\begin{theorem}\label{ci-affine} Let $\mathcal{C}$ be a clutter with $s$ edges and let
$T=\{(x_1,\ldots,x_s)\vert\, x_i\in K^*\mbox{ for all }i\}$ be an
affine torus in $\mathbb{A}^s$. The following are
equivalent\/{\rm:}
\begin{itemize}
\item[($\mathrm{a}_1$)] $I(Y)$ is a complete intersection.
\item[($\mathrm{a}_2$)]
$I(Y)=(t_1^{q-1}-t_{s+1}^{q-1},\ldots,t_s^{q-1}-t_{s+1}^{q-1})$.
\item[($\mathrm{a}_3$)] $X^*=T$.
\item[($\mathrm{a}_4$)] $I(X^*)=(t_1^{q-1}-1,\ldots,t_s^{q-1}-1)$.
\end{itemize}
\end{theorem}

\begin{proof} ($\mathrm{a}_1$)$\Rightarrow$($\mathrm{a}_2$): It
follows at once from Lemma~\ref{1-jun-11}.
($\mathrm{a}_2$)$\Rightarrow$($\mathrm{a}_3$): By
Proposition~\ref{ci-summary} one has 
$I(Y)=I(\mathbb{T}')=(\{t_i^{q-1}-t_{s+1}^{q-1}\}_{i=1}^{s})$, where
$\mathbb{T}'$ is a projective torus in $\mathbb{P}^s$. As
$Y$ and $\mathbb{T}'$ are both projective varieties, we get that
$Y=\mathbb{T}'$ (see \cite[Lemma 4.2]{algcodes}). We need only show
the inclusion $T\subset X^*$ . Take $a$ in $T$. Then,
$[(a,1)]\in\mathbb{T}'=Y$. Thus, we get $a\in X^*$.
($\mathrm{a}_3$)$\Rightarrow$($\mathrm{a}_4$): We need only show the
inclusion ``$\subset$''. Take $f\in I(X^*)$. By the division algorithm
\cite[Theorem~1.5.9, p.~30]{AL} we can
write 
\begin{equation*}
f=h_1(t_1^{q-1}-1)+\cdots+ h_s(t_s^{q-1}-1)+g,
\end{equation*}
for some $h_1,\ldots,h_s,g$ in $S$, where the monomials that occur 
in $g$ are not divisible by any of the monomials $t_1^{q-1},\ldots,t_s^{q-1}$, i.e.,
$\deg_{t_i}(g)<q-1$ for $i=1,\ldots,s$. Hence, since $g$ vanishes on
all $T$, using the Combinatorial Nullstellensatz
\cite[Theorem~1.2]{alon-cn} it follows readily that $g=0$, that is,
$f\in(\{t_i^{q-1}-1\}_{i=1}^s)$.
($\mathrm{a}_4$)$\Rightarrow$($\mathrm{a}_1$): Let $\succ$ be the
{\it elimination order\/} on the  
monomials of $S[u]$, where $u=t_{s+1}$. Recall that 
this order is defined as  $t^b\succ t^a$ if the degree of $t^b$ is
greater than that of $t^a$, or both degrees are equal, and the last
nonzero component  of $b-a$ is negative. As $K$ is a finite field,
$Y$ is the projective closure of $X^*$, i.e.,
$\overline{X^*}=\overline{Y}=Y$. Since
$t_1^{q-1}-1,\ldots,t_s^{q-1}-1$ form a Gr\"obner basis with respect
to $\succ$, using \cite[Proposition~2.4.30]{monalg}, we get the
equality $I(Y)=(\{t_i^{q-1}-t_{s+1}^{q-1}\}_{i=1}^s)$. Thus
$I(Y)$ is a complete intersection.
\end{proof}

\begin{corollary}\label{jun19-11-1} Let $\mathcal{C}$ be a clutter. If $I(Y)$ is a
complete intersection, then $I(X)$ is a complete intersection.
\end{corollary}

\begin{proof} By Theorem~\ref{ci-affine}, $X^*$ is an affine torus.
Then, $X$ is a projective torus. Consequently $I(X)$ is a complete
intersection by Proposition~\ref{ci-summary}. 
\end{proof}

\begin{corollary}\label{jun-19-11-2} Let $G$ be a graph. If
$\gcd(q-1,2)=1$ or if $G$ is 
bipartite, then $I(Y)$ is a complete intersection if and only if
$I(X)$ is a complete 
intersection. 
\end{corollary}

\begin{proof} It follows at once from
Corollary~\ref{sunday-jun19-11-night} and Corollary~\ref{jun19-11-1}.
\end{proof}

The converse of Corollary~\ref{jun19-11-1} is not true as the next example shows.

\begin{example}\label{jun24-11} Let $X$ be the projective algebraic
toric set parameterized 
$y_1y_2,y_2y_3,y_1y_3$ and let $K=\mathbb{F}_5$. Then,
$I(X)=(t_1^4-t_3^4, t_2^4-t_3^4)$ is a complete intersection but 
$$
I(Y)=(t_3^4-t_4^4, t_2^2t_3^2-t_1^2t_4^2, t_1^2t_3^2-t_2^2t_4^2, t_2^4-t_4^4,
      t_1^2t_2^2-t_3^2t_4^2, t_1^4-t_4^4)
$$
is not a complete intersection. The generators of $I(Y)$ were computed
using the computer algebra system {\em Macaulay\/}$2$ \cite{mac2} and the methods
of \cite{affine-codes,algcodes}. If $K=\mathbb{F}_4$, then $I(X)$ and
$I(Y)$ are both complete intersections in concordance with
Corollary~\ref{jun-19-11-2}.
\end{example}

\begin{proposition}\label{jul27-11} If $G$ is a connected graph, then $I(X)$ is a complete
intersection if and only if $G$ is a tree or $G$ is a unicyclic graph with a unique
odd cycle.
\end{proposition}

\begin{proof} $\Rightarrow$) As $I(X)$ is a complete intersection,
$X\subset\mathbb{P}^{s-1}$ is a projective torus \cite[Corollary~4.5]{ci-codes}. 
Thus, $|X|=(q-1)^{s-1}$. If $G$ is bipartite, 
then $|X|=(q-1)^{n-2}$ \cite[Corollary 3.8]{algcodes}.  
Hence, $s=n-1$ and $G$ is a tree because $G$ is connected. If $G$
is not bipartite, then $|X|=(q-1)^{n-1}$ 
\cite[Corollary 3.8]{algcodes}. Thus, $s=n$ and $G$ is a unicyclic graph.

$\Leftarrow$) Let $\mathbb{T}$ be a projective torus in 
$\mathbb{P}^{s-1}$. If $G$ is a tree, then $s=n-1$ and
$|X|=(q-1)^{n-2}$ \cite[Corollary~3.8]{algcodes}. Since $X\subset\mathbb{T}$ and
$|\mathbb{T}|=(q-1)^{s-1}$, we get that $|X|=|\mathbb{T}|$. Thus,
$X=\mathbb{T}$. Consequently, $I(X)$ is a complete intersection by
Proposition~\ref{ci-summary}. If $G$ is a unicyclic graph with a unique
odd cycle, then $s=n$ and $|X|=(q-1)^{n-1}$
\cite[Corollary~3.8]{algcodes}. 
Since $X\subset\mathbb{T}$ and
$|\mathbb{T}|=(q-1)^{s-1}$, we get that $|X|=|\mathbb{T}|$. Thus,
$X=\mathbb{T}$. Hence, $I(X)$ is a complete intersection by
Proposition~\ref{ci-summary}. 
\end{proof}

From this result it follows that for connected graphs, with $q\geq 3$, 
the complete intersection property of $I(X)$
is independent of the finite field $K$. The complete intersection
property of $I(Y)$ depends on 
the finite field $K$ as seen in Example \ref{jun24-11}. The following 
result describes when $I(Y)$ is a complete intersection for connected graphs.

\begin{theorem}\label{main-iy-graphs} Let $G$ be a connected graph. Then 
$I(Y)$ is a complete intersection if and only if $G$ is a
tree or $G$ is a unicyclic graph with a unique odd cycle and $q$ is even.
\end{theorem}

\begin{proof} $\Rightarrow$) By Corollary~\ref{jun19-11-1}, $I(X)$ is
a complete intersection. Then, by Proposition~\ref{jul27-11}, 
$G$ is a tree or $G$ is a unicyclic graph with a unique
odd cycle. If $G$ is a tree, there is nothing to prove. Assume that $G$
is not a tree. Then, $s=n$. Notice that in general $|X^*|=|Y|$. 
If $q$ is odd, then by Corollary~\ref{maria-cafeteria-cinvestav-1} and
Theorem~\ref{ci-affine}, we get:
$$
|Y|=(q-1)^n/2\ \mbox{ and }\ |Y|=|X^*|=(q-1)^s=(q-1)^n,
$$
a contradiction. Thus, $q$ is even, as required. 

$\Leftarrow$) It follows readily from Proposition~\ref{jul27-11} and
Corollary~\ref{jun-19-11-2}. 
\end{proof}

\begin{corollary} Let $G$ be a connected bipartite graph. Then 
$I(Y)$ is a complete intersection if and only if $G$ is a tree. 
\end{corollary}

\bibliographystyle{plain}

\end{document}